\documentclass[12pt]{article}

\usepackage{amssymb,amsthm,amsmath}
\usepackage[utf8]{inputenc}
\usepackage{verbatim}
\usepackage{geometry}

\newcommand{\R}{\mathbb{R}}

\newcommand{\1}{\textbf{1}}
\newcommand{\dd}{\mathrm{d}}

\newcommand{\call}[4]{\int_{#1}^{#2} {#3} \; \textrm{d} {#4}}

\newcommand{\fun}[3]{#1\colon #2 \longrightarrow #3}

\newcommand{\mb}[1]{\mathbb{{#1}}}

\newcommand{\mc}[1]{\mathcal{{#1}}}
\newcommand{\e}{\varepsilon}

\newcommand{\vp}{\varphi}

\DeclareMathOperator{\var}{\textrm{Var}}
\DeclareMathOperator{\dist}{\textrm{dist}}
\DeclareMathOperator{\med}{\textrm{Med}}
\DeclareMathOperator{\ent}{\textrm{Ent}}

\newtheorem{prop}{Proposition}
\newtheorem{thm}{Theorem}

\newtheorem{lem}{Lemma}
\newtheorem{cor}{Corollary}

\theoremstyle{definition}
\newtheorem*{def*}{Definition}
\newtheorem{definition}{Definition}

\theoremstyle{remark}
\newtheorem*{rem*}{Remark}

\title{A note on the convex infimum convolution inequality}

\author{Naomi Feldheim, Arnaud Marsiglietti, Piotr Nayar, Jing Wang}
\date{}

\newgeometry{tmargin=2cm, bmargin=2cm, lmargin=2cm, rmargin=2cm}

\begin{document}

\maketitle

\begin{abstract}
We characterize the symmetric measures which satisfy the one dimensional convex infimum convolution inequality of Maurey. For these measures the tensorization argument yields the two level Talagrand's concentration inequalities for their products and convex sets in $\R^n$.   
\end{abstract}



\section{Introduction}\label{sec:intro}

In the past few decades a lot of attention has been devoted to study the concentration of measure phenomenon, especially the concentration properties of product measures on the Euclidean space $\mb{R}^n$. Through this note by $|x|_p$ we denote the $l_p$ norm on $\mb{R}^n$, namely $|x|_p=(\sum_{i=1}^n|x_i|^p)^{1/p}$, and let us take $B_p^n=\{x \in \mb{R}^n: \ |x|_p \leq 1 \}$. We say that a Borel probability measure $\mu$ on $\mb{R}^n$ satisfies concentration with a profile $\alpha_\mu(t)$ if for any set $A \subset \mb{R}^n$ with $\mu(A) \geq 1/2$ we have $\mu(A+tB_2^n) \geq 1-\alpha_\mu(t)$, $t \geq 0$, where $B_2^n$ is the Euclidean ball of radius $1$, centred at the origin. Equivalently, for any $1$-Lipschitz function $f$ on $\mb{R}^n$ we have $\mu(\{x: \ |f-\med_\mu f| > t\}) \leq \alpha_\mu(t)$, $t >0$, where $\med_\mu f$ is a median of $f$. Moreover, the median can be replaced with the mean, see \cite[Proposition 1.8]{L}.  

The usual way to reach concentration is via certain functional inequalities. For example, we say that $\mu$ satisfies  Poincar\'e inequality (sometimes also called the spectral gap inequality) with constant $C$, if for any $f:\mb{R}^n \to \mb{R}$ which is, say, $C^1$ smooth, we have 
\begin{equation}\label{poincare}
	\var_\mu(f) \leq C\call{}{}{|\nabla f|^2}{\mu}.
\end{equation}
We assume that $C$ is the best constant possible.
This inequality implies the exponential concentration, namely, concentration with a profile $\alpha_\mu(t)=2\exp(-t/2\sqrt{C})$. The product exponential probability measure $\nu_n$, i.e., measure with density $2^{-n}\exp(-|x|_1)$, satisfies \eqref{poincare} with the constant $4$ (the one dimensional case goes back to \cite{K}; for a one-line proof see Lemma 2.1 in \cite{BL}). From this fact (case $n=1$) one can deduce, using the transportation argument, that any one-dimensional log-concave measure, i.e. the measure with the density of the form $e^{-V}$, where $V$ is convex, satisfies the Poincar\'e inequality with a universal constant $C$. In fact, there is a so-called Muckenhoupt condition that fully characterize measures satisfying the one dimensional Poincar\'e inequality, see \cite{Mu} and \cite{K}. Namely, let us assume, for simplicity, that our measure is symmetric. If the density of the absolutely continuous part of $\mu$ is equal to $p$ then the optimal constant in \eqref{poincare} satisfies
\[
	\frac{1}{(1+\sqrt{2})^2} B \leq C \leq 4  B, \qquad \textrm{where} \qquad  B = \sup_{x>0} \left\{\mu[x,\infty) \call{0}{x}{\frac{1}{p(y)}}{y} \right\} .
\]
In particular, the constant $C$ is finite if $B < \infty$. Suppose that $\mu_1,\ldots \mu_1$ satisfy Poincar\'e inequality with constant $C$. Then the same can be said about the product measure $\mu= \mu_1 \otimes \ldots \otimes \mu_n$, due to tensorization property of \eqref{poincare} (see Corollary 5.7 in \cite{L}), which follows immediately from the subadditivity of the variance. Thus, one can say that the Poincar\'e inequality is fully understood in the case of product measures.

A much stronger functional inequality that gives concentration is the so-called log-Sobolev inequality,
\begin{equation}\label{sobolev}
	\ent_\mu(f^2) \leq 2C\call{}{}{|\nabla f|^2}{\mu},
\end{equation}
where for a function $g \geq 0$ we set
\[
\ent_\mu(g)=\call{}{}{g \ln g}{\mu}-\left(\call{}{}{g}{\mu} \right) \ln \left( \call{}{}{g}{\mu} \right).
\]
It implies the Gaussian concentration phenomenon, i.e., concentration with a profile $\alpha_\mu(t)=\exp(-t^2/16C)$. As an example, the standard Gaussian measure $\gamma_n$, i.e., the measure with density $(2\pi)^{-n/2} e^{-|x|^2/2}$, satisfies $\eqref{sobolev}$ with the constant $1$ (see \cite{G}). As in \eqref{poincare}, the log-Sobolev inequality possess similar terrorization properties and there is a full description of measures $\mu$ satisfying \eqref{sobolev} on the real line, see \cite{BG2}. Namely, the optimal constant $C$ in \eqref{sobolev} satisfies
\[
	\frac{1}{75} B' \leq C \leq 936  B', \quad \textrm{where} \quad  B' = \sup_{x>0} \left\{ \mu[x,\infty) \ln \left( \frac{1}{\mu[x,\infty)} \right) \call{0}{x}{\frac{1}{p(y)}}{y} \right\}.
\]

Another way to concentration leads through the so-called property $(\tau)$ (see \cite{M} and \cite{LW}). 
A measure $\mu$ on $\R^n$ is said to satisfy property $(\tau)$ with a nonnegative cost function $\vp$ if the inequality
\begin{equation}\label{tau}
	\left( \call{}{}{e^{f \square \vp}}{\mu} \right)\left( \call{}{}{e^{-f }}{\mu} \right) \leq 1
\end{equation}   
holds for every bounded measurable function $f$ on $\R^n$. Here $f \square \vp = \inf_{y}\{ f(y)+\vp(x-y) \}$ is the so-called infimum convolution. Property $(\tau)$ implies concentration, see Lemma 4 in \cite{M}. Namely, for every measurable set $A$ we have
\[
	\mu\left(\{ x \notin A + \{ \vp < t \} \} \right) \leq \mu(A)^{-1} e^{-t}.
\]   

In the seminal paper \cite{M} Maurey showed that the exponential measure $\nu_1$ satisfies the infimum convolution inequality with a the cost function $\vp(t)=\min\{\frac{1}{36}t^2, \frac{2}{9}(|t|-2)\}$. As a consequence of tensorization (see Lemma 1 in \cite{M}) we get that $\nu_n$ satisfies \eqref{tau} with the cost function $\vp(x)=\sum_{i=1}^n \vp(x_i)$. This leads to the so-called Talagrand's two level concentration,
\[
	\nu_n\left( A + 6\sqrt{t} B_2^n + 9t B_1^n \right) \geq 1- \frac{1}{\nu_n(A)}e^{-t},
\]
(see Corollary 1 in \cite{M}). As we can see, this gives a stronger conclusion than the Poincar\'e inequality. It is a well known fact that the property $(\tau)$ with a cost function which is quadratic around $0$, say $\vp(x)=C|x|^2$ for $|x|\leq 1$, implies the Poincar\'e inequality with the constant $1/(4C)$, see Corrolary 3 in \cite{M}. We shall sketch the argument in Section \ref{sec2}.

The goal of this article is to investigate concentration properties of a wider class of measures, i.e., measures that may not even satisfy the Poincar\'e inequality. For example, in the case of purely atomic measures one can easily construct a non-constant function $f$ with $\call{}{}{|\nabla f|^2}{\mu}=0$. However, one can still hope to get concentration if one restricts himself to the class of convex sets. It turns out that to reach exponential concentration for convex sets, it suffices to prove that the measure $\mu$ satisfies the convex Poincar\'e inequality. Below we state the formal definition.
\begin{definition}
We say that a Borel probability measure $\mu$ on $\R$ satisfies the convex Poincar\'e inequality with a constant $C_p$ if for every convex function $f:\R \to \R$ with $f'$ bounded we have
\begin{equation}\label{main}
	\var_\mu f \leq C_p \call{\R}{}{(f')^2}{\mu}.
\end{equation}
\end{definition}
\noindent Here we adopt the standard convention that for a locally Lipschitz function $\fun{f}{\R}{\R}$ the gradient  $f'$ is defined by
\begin{equation}\label{grad}
	f'(x) = \limsup_{h\to0} \frac{f(x+h) - f(x)}{h}.
\end{equation}
This definition applies in particular to convex $f$. If $f$ is differentiable, \eqref{grad} agrees with the usual derivative. 

We also  need the following definition.
\begin{definition}
Let $h>0$ and $\lambda \in [0,1)$. Let $\mc{M}(h,\lambda)$ be the class of symmetric measures on $\mb{R}$, satisfying the condition $\mu[x+h,\infty) \leq \lambda \mu[x,\infty)$ for $x \geq 0$. Moreover, let $\mc{M}^+(h,\lambda)$ be the class of measures supported on $\mb{R_+}$, satisfying the same condition.
\end{definition}
The inequality \eqref{main} has been investigated by Bobkov and G\"otze, see \cite[Theorem 4.2]{BG}. In particular, the authors proved \eqref{main} in the class $\mc{M}(h,\lambda)$ with a constant $C_p$ depending only on $h$ and $\lambda$. This leads to the exponential concentration for $1$-Lipschitz convex functions $f$ (as well as the exponential concentration for convex sets) via the standard Herbst-like argument which actually goes back to \cite{AS} (see also, e.g., Theorem 3.3 in \cite{L}), by deriving the bound on the Laplace transform of $f$.

It is worth mentioning that a much stronger condition $\mu[x+C/x,\infty) \leq \lambda\mu[x,\infty)$, $x \geq m$, where $m$ is a fixed positive number, has been considered. It implies log-Sobolev inequalities for log-convex functions and the Gaussian concentration of measure for convex sets. We refer to the nice study \cite{A} for the details.

In \cite{M} the author showed that every measure supported on a subset of $\R$ with diameter $1$ satisfies the convex property $(\tau)$, i.e. the inequality \eqref{tau} for convex functions $f$, with the cost function $\frac14 |x|^2$ (see Theorem 3 therein). The aim of this note is to extend both this result and the convex Poincar\'e inequality of Bobkov and G\"otze. We introduce the following definition.
\begin{definition}
Define 
\[
	\vp_0(x)= \left\{  \begin{array}{ll} \frac12 x^2 & |x| \leq 1 \\
	|x|-\frac12 & |x| > 1
	
	\end{array} \right. .
\]
We say that a Borel probability measure $\mu$ on $\mb{R}$ satisfies convex exponential property $(\tau)$ with constant $C_\tau$ if for any convex function $f:\mb{R} \to \mb{R}$ with $\inf f>-\infty$ we have
\begin{equation}\label{deftau}
	\left( \call{}{}{e^{f \Box \vp}}{\mu} \right) \left( \call{}{}{e^{-f}}{\mu} \right) \leq 1
\end{equation}	
with $\vp(x)=\vp_0(x/C_\tau)$.
\end{definition}

\noindent We are ready to state our main result.

\begin{thm}\label{convtau}

The following conditions are equivalent
\begin{itemize}
\item[(a)] $\mu \in \mc{M}(h,\lambda)$,
\item[(b)] There exists $C_p>0$ such that $\mu$ satisfies the convex Poincar\'e inequality with constant $C_p$, 
\item[(c)] There exists $C_\tau>0$ such that $\mu$ satisfies the convex exponential property $(\tau)$ with constant $C_\tau$. 
\end{itemize}
Moreover, (a) implies (c) with the constant $C_\tau=17 h/(1-\lambda)^2$, (c) implies (b) with the constant $C_p=\frac12 C_\tau^2$ and (b) implies (a) with $h= \sqrt{8C_p}$ and $\lambda= 1/2$.
\end{thm}

\noindent This generalizes Maurey's theorem due to the fact that any symmetric measure supported on the interval $[-1,1]$  belongs to $\mc{M}(1,0)$. 

It is well known that the convex property $(\tau)$ tensorizes, namely, if $\mu_1,\ldots, \mu_n$ have convex property $(\tau)$ with cost functions $\vp_1,\ldots \vp_n$ then the product measure $\mu=\mu_1 \otimes \ldots \otimes \mu_n$ has convex property $(\tau)$ with $\vp(x)=\sum_{i=1}^n \vp_i(x_i)$, see \cite[Lemma 5]{M}. Therefore Theorem \ref{convtau} implies the following Corollary.

\begin{cor}\label{tensor}
Let $\mu_1,\ldots, \mu_n \in \mc{M}(h,\lambda)$ and let us take $\mu=\mu_1 \otimes \ldots \otimes \mu_n$. Define the cost function $\vp(x)=\sum_{i=1}^n \vp_0(x_i/C_\tau)$, where $C_\tau=17h/(1-\lambda)^2$. Then for any coordinate-wise convex function $f$ (in particular, for any convex function) we have
\begin{equation}\label{dimn}
	\left( \call{}{}{e^{f \square \vp}}{\mu} \right)\left( \call{}{}{e^{-f }}{\mu} \right) \leq 1.
\end{equation}
\end{cor}

\noindent As a consequence, one can deduce the two-level concentration for convex sets and convex functions in $\mb{R}^n$. 

\begin{cor}\label{corr1}
Let $\mu_1,\ldots, \mu_n \in \mc{M}(h,\lambda)$. Let $C_\tau = 17h/(1-\lambda)^2$. Take $\mu=\mu_1 \otimes \ldots \otimes \mu_n$. Then for any convex set $A$ with $\mu(A)>0$ we have
\[
	\mu\left(A+\sqrt{2t}C_\tau B_2^n + 2t C_\tau B_1^n \right) \geq 1- \frac{e^{-t}}{\mu(A)}.
\] 
\end{cor}

\begin{cor}\label{corr2}
Let $\mu_1,\ldots, \mu_n \in \mc{M}(h,\lambda)$. Let $C_\tau = 17h/(1-\lambda)^2$. Take $\mu=\mu_1 \otimes \ldots \otimes \mu_n$. Then for any convex function $f:\mb{R}^n \to \mb{R}$ with
\[
	|f(x)-f(y)|_2 \leq a |x-y|_2, \qquad |f(x)-f(y)|_1 \leq b |x-y|_1, \quad x,y \in \mb{R}^n,   
\]
we have
\begin{equation}\label{ineq1}
	\mu\left( \left\{  f > \med_\mu f +  C_\tau t  \right\} \right)  \leq 2 \exp\left( -\frac18 \min\left\{\frac{t}{b}, \frac{t^2}{a^2}\right\} \right), \qquad t \geq 0,
\end{equation}
and
\begin{equation}\label{ineq2}
	\mu\left( \left\{  f < \med_\mu f -  C_\tau t   \right\} \right)  \leq 2 \exp\left( -\frac18 \min\left\{\frac{t}{b}, \frac{t^2}{a^2}\right\} \right), \qquad t \geq 0.
\end{equation}
\end{cor}

Let us mention that very recently Adamczak and Strzelecki established related results in the context of modified log-Sobolev inequalities, see \cite{ASt}. For simplicity we state their result in the case of symmetric measures. For $\lambda \in [0,1)$, $\beta \in [0,1]$ and $h,m>0$ the authors defined the class of measures $\mc{M}_{AS}^\beta(h,\lambda,m)$ satisfying the condition $\mu([x+h/x^\beta,\infty) \leq \lambda \mu([x,\infty))$ for $x \geq m$. Note that $\mc{M}_{AS}^0(h,\lambda,0)=M(h,\lambda)$. They proved that products $\mu=\mu_1\otimes \ldots \otimes \mu_n$ of such measures satisfy the inequality
\begin{equation}\label{AS}
	\ent_\mu(e^f) \leq C_{AS} \call{}{}{e^f \left( |\nabla f|^2 \lor |\nabla f |_{\scriptscriptstyle{\frac{\beta+1}{\beta}}}^{\frac{\beta+1}{\beta}} \right)  }{\mu}. 
\end{equation}
for any smooth convex function $f:\mb{R} \to \mb{R}$. As a consequence, for any convex set $A$ in $\mb{R}^n$ with $\mu(A) \geq 1/2$ we have
\[
	\mu\left(A+ t^{\frac12} B_2^n + t^\frac{1}{1+\beta} C_\tau B_{1+\beta}^n \right) \geq 1- e^{-C_{AS}'t}, \quad t \geq 0.
\]
Here the constants $C_{AS},C_{AS}'$ depend only on the parameters $\beta,m,h$ and $\lambda$. They also established inequality similar to \eqref{ineq1}, namely for a convex function $f$ with  
\[
	|f(x)-f(y)|_2 \leq a |x-y|_2, \qquad |f(x)-f(y)|_{1+\beta} \leq b |x-y|_{1+\beta}, \quad x,y \in \mb{R}^n,   
\]
one gets
\[
\mu\left( \left\{  f > \med_\mu f +  2t  \right\} \right)  \leq 2 \exp\left( -\frac{3}{16} \min\left\{\frac{t^{1+ \beta}}{b^{1+\beta} C_{AS}^\beta}, \frac{t^2}{a^2 C_{AS}}\right\} \right), \qquad t \geq 0.
\]
However, the authors were not able to get \eqref{ineq2}. In fact one can show that for $\beta=0$ our Theorem \ref{convtau} is stronger than \eqref{AS}. In particular, the inequality \eqref{AS} is equivalent to 
\[
	\call{}{}{e^{f \Box \vp}}{\mu} \leq e^{\call{}{}{f}{\mu}},
\]
see \cite{AS}, which easily follows from \eqref{dimn}.

The rest of this article is organized as follows. In the next section we prove Theorem \ref{convtau}. In Section \ref{corro} we deduce Corollaries \ref{corr1} and \ref{corr2}.

\section{Proof of Theorem \ref{convtau}}\label{sec2}

We need the following lemma, which is essentially included in Theorem 4.2 of \cite{BG}. For reader's convenience we provide a straightforward proof of this fact.
 
\begin{lem}\label{bob}
Let $\mu \in \mc{M}^+(h, \lambda)$ and let $g:\mb{R} \to [0,\infty)$ be non-decreasing with $g(0)=0$. Then
\[
	\call{}{}{g^2}{\mu} \leq  \left(\frac{2}{1-\lambda}\right)^2  \call{}{}{  (g(x)-g(x-h))^2  }{\mu(x)} . 
\]
\end{lem}

\begin{proof}
We first prove that
\[
	\lambda \call{}{}{g(x)}{\mu(x)} \geq \call{}{}{g(x-h)}{\mu(x)}
\]
for any non-decreasing $g:\mb{R} \to [0,\infty)$ such that $g(0)=0$. Both sides of this inequality are linear in $g$. Therefore, it is enough to consider only functions of the form $g(x)=\1_{[a,\infty)}(x)$ for $a \geq 0$, since $g$ can be expressed as a mixture of these functions. For $g(x)=\1_{[a,\infty)}(x)$ the above inequality reduces to
$\lambda \mu[a,\infty) \geq \mu[a+h,\infty)$, 
which is clearly true due to our assumption on $\mu$.

The above is equivalent to
\begin{equation}\label{wq}
	(1-\lambda) \call{}{}{g(x)}{\mu(x)} \leq	\call{}{}{(g(x)-g(x-h))}{\mu(x)} .
\end{equation}
Now, let us use \eqref{wq} with $g^2$ instead of $g$. Then,
\begin{align*}
	\call{}{}{g^2}{\mu} & \leq \frac{1}{1-\lambda} 	\call{}{}{(g(x)^2-g(x-h)^2)}{\mu(x)} \\
	& = \frac{1}{1-\lambda} 	\call{}{}{(g(x)-g(x-h))(g(x)+g(x-h))}{\mu(x)} \\
	& \leq \frac{1}{1-\lambda} 	\left(\call{}{}{(g(x)-g(x-h))^2}{\mu(x)} \right)^{1/2} 	\left(\call{}{}{(g(x)+g(x-h))^2}{\mu(x)} \right)^{1/2} \\
	& \leq \frac{2}{1-\lambda} 	\left(\call{}{}{(g(x)-g(x-h))^2}{\mu(x)} \right)^{1/2} 	\left(\call{}{}{g(x)^2}{\mu(x)} \right)^{1/2}
\end{align*}
We arrive at
\[
	\left( \call{}{}{g^2}{\mu} \right)^{1/2} \leq \frac{2}{1-\lambda} \left(\call{}{}{(g(x)-g(x-h))^2}{\mu(x)} \right)^{1/2}. 
\]
Our assertion follows.
\end{proof}

In the rest of this note, we take $f:\mb{R} \to \mb{R}$ to be convex. Let $x_0$ be a point where $f$ attains its minimal value. Note that this point may not be unique. However, one can check that what follows does not depend on the choice of $x_0$. Moreover, if $f$ is increasing (decreasing) we adopt the notation $x_0=-\infty$ ($x_0=\infty$). Let us define a discrete version of gradient of $f$,
\[
	(Df)(x) = \left\{ \begin{array}{ll}
	f(x)-f(x-h) & x > x_0+h \\
	f(x)-f(x_0) & x \in [x_0-h,x_0+h] \\
	f(x)-f(x+h) & x < x_0-h
\end{array}	 \right. .
\] 

\begin{lem}\label{bobex}
Let $f:\mb{R} \to \mb{R}$ be a convex function with $f(0)=0$ and let $\mu \in \mc{M}(h,\lambda)$. Then 
\[
	 \call{}{}{\left(e^{f/2} - e^{-f/2}  \right)^2}{\mu} \leq \frac{8}{(1-\lambda)^2} \call{}{}{ e^{f(x)} ((Df)(x))^2  }{\mu(x)}.
\]
\end{lem}

\begin{proof}
\textit{Step 1.} We first assume that $f$ is non-negative and non-decreasing. It follows that $f(x)=0$ for $x \leq 0$.  Correspondingly, the function $g=e^{f/2}-e^{-f/2}$ is non-negative, non-decreasing and $g(0)=0$. Moreover, let $\mu^+$ be the normalized restriction of $\mu$ to $\mb{R}_+$. Note that $\mu^+ \in \mc{M}^+(h,\lambda)$. From Lemma \ref{bob} we get
\begin{align*}
	\call{}{}{\left(e^{f/2} - e^{-f/2}  \right)^2}{\mu} & = \frac12 \call{}{}{g^2}{\mu^+} \leq \frac{2}{(1-\lambda)^2} \call{}{}{  (g(x)-g(x-h))^2  }{\mu^+(x)} \\
	& =  \frac{4}{(1-\lambda)^2} \call{}{}{  (g(x)-g(x-h))^2  }{\mu(x)}.
\end{align*}
Observe that
\begin{align*}
	g(x)-g(x-h) & = e^{\frac{f(x)}{2}} - e^{-\frac{f(x)}{2}} -e^{\frac{f(x-h)}{2}}  + e^{-\frac{f(x-h)}{2}} \\
	& = \left( e^{\frac{f(x)}{2}}-e^{\frac{f(x-h)}{2}} \right)  \left( 1 +e^{-\frac{f(x)}{2} -\frac{f(x-h)}{2}} \right) \\
	& \leq 2\left(e^{\frac{f(x)}{2}}-e^{\frac{f(x-h)}{2}} \right) \leq  e^{\frac{f(x)}{2}} (f(x)-f(x-h)),
\end{align*}
where the last inequality follows from the mean value theorem. We arrive at
\[
	 \call{}{}{\left(e^{f/2} - e^{-f/2}  \right)^2}{\mu} \leq  \left(\frac{2}{1-\lambda}\right)^2 \call{}{}{ e^{f(x)} ((Df)(x))^2  }{\mu(x)}.
\]

\textit{Step 2.} Now let $f$ be non-decreasing but not necessarily non-negative. From convexity of $f$ and the fact that $f(0)=0$ we get $|f(-x)| \leq f(x)$ for $x \geq 0$. This implies the inequality $|e^{f(-x)}-e^{-f(-x)}| \leq |e^{f(x)}-e^{-f(x)}|$, $x \geq 0$. From the symmetry of $\mu$ one gets
\[
	\call{}{}{\left(e^{f/2} - e^{-f/2}  \right)^2}{\mu} \leq \call{}{}{\left(e^{f/2} - e^{-f/2}  \right)^2}{\mu^+}.
\]
Let $\tilde{f}=f\1_{[0,\infty)}$. From Step 1 one gets
\begin{align*}
	\call{}{}{\left(e^{f/2} - e^{-f/2}  \right)^2}{\mu^+} & = 2 \call{}{}{\left(e^{\tilde{f}/2} - e^{-\tilde{f}/2}  \right)^2}{\mu} \\
	& \leq \frac{8}{(1-\lambda)^2} \call{}{}{ e^{\tilde{f}(x)} ((D\tilde{f})(x))^2  }{\mu(x)} \\
	& \leq \frac{8}{(1-\lambda)^2} \call{}{}{ e^{f(x)} ((Df)(x))^2  }{\mu(x)}.
\end{align*}

\textit{Step 3.} The conclusion of Step 2 is also true in the case of non-increasing functions with $f(0)=0$. This is due to the invariance of our assertion under the symmetry $x \to -x$, which is an easy consequence of the symmetry of $\mu$ and the fact that for $F(x)=f(-x)$ we have $(DF)(x) = (Df)(-x)$.

\textit{Step 4.} Let us now eliminate of the assumption of monotonicity of $f$. Suppose that $f$ is not monotone. Then $f$ has a (not necessarily unique) minimum attained at some point $x_0\in \mb{R}$. Due to the remark of Step 3 we can assume that $x_0 \leq 0$. Since $f(0)=0$, we have $f(x_0)<0$. Take $y_0=\inf\{y \in \mb{R}: \ f(y)=0\}$. Clearly $y_0\leq x_0$. We define
\[
	f_1(x) = \left\{ \begin{array}{ll}
	f(x) & x \geq x_0 \\
	f(x_0) & x < x_0 
\end{array}	 \right. , \qquad
f_2(x) = \left\{ \begin{array}{ll}
	0 & x \geq y_0 \\
	f(x) & x < y_0 
\end{array}	 \right. .
\]
Note that $f_1$ is non-decreasing and $f_2$ is non-increasing. Moreover, $f_1(0)=f_2(0)=0$. Therefore, from the previous steps applied for $f_1$ and $f_2$ we get
\begin{equation}\label{eqq1}
(e^{f(x_0)/2}-e^{-f(x_0)/2})^2 \mu((-\infty,x_0]) + \call{x_0}{+\infty}{\left(e^{f/2} - e^{-f/2}  \right)^2}{\mu}  \leq  \frac{8}{(1-\lambda)^2} \call{x_0}{+\infty}{ e^{f(x)} ((Df)(x))^2  }{\mu(x)} 
\end{equation}
and
\begin{equation}\label{eqq2}
\call{-\infty}{y_0}{\left(e^{f/2} - e^{-f/2}  \right)^2}{\mu}	\leq  \frac{8}{(1-\lambda)^2} \call{-\infty}{y_0}{ e^{f(x)} ((Df)(x))^2  }{\mu(x)} .
\end{equation}
Moreover, since $|f(x)| \leq |f(x_0)|$ on $[y_0,x_0]$, we have
\begin{equation}\label{eqq3}
	\call{y_0}{x_0}{\left(e^{f/2} - e^{-f/2}  \right)^2}{\mu} \leq (e^{f(x_0)/2}-e^{-f(x_0)/2})^2 \mu([y_0,x_0]) \leq  (e^{f(x_0)/2}-e^{-f(x_0)/2})^2 \mu((-\infty,x_0]).
\end{equation}
Combining \eqref{eqq1}, \eqref{eqq2} and \eqref{eqq3}, we arrive at
\[
	\call{}{}{\left(e^{f/2} - e^{-f/2}  \right)^2}{\mu} \leq \frac{8}{(1-\lambda)^2} \call{}{}{ e^{f(x)} ((Df)(x))^2  }{\mu(x)} .
\]
\end{proof}

\noindent The following lemma provides an estimate on the infimum convolution.
\begin{lem}\label{bound}
Let $C_1,h>0$. Define $\vp_1(x)=\frac{1}{C_1}\vp_0(x/h)$. Assume that a convex function $f$ satisfies $|f'| \leq 1/(C_1 h)$. Then 
\[
	(f \Box \vp_1)(x) \leq f(x) - \frac{C_1}{2} |(Df)(x)|^2.
\]
\end{lem}
\begin{proof}
Let us consider the case when $x \geq x_0+h$. We take $\theta \in [0,1]$ and write $y=\theta (x-h) + (1-\theta)x$. Note that $x-y=h\theta$. By the convexity of $f$ we have
\begin{align*}
	(f \Box \vp_1)(x) & \leq f(y)+\vp_1(x-y) \leq \theta f(x-h)+(1-\theta)f(x) + \vp_1(h\theta) \\
	& = \theta f(x-h)+(1-\theta)f(x) + \frac{1}{2C_1}\theta^2.
\end{align*}
Let us now take $\theta = C_1(f(x)-f(x-h))$. Note that $0 \leq f' \leq 1/C_1h$ yields $\theta \in [0,1]$. We get
\[
	(f \Box \vp_1)(x) \leq f(x) - \theta(f(x)-f(x-h)) +  \frac{1}{2C_1}\theta^2 = f(x)- \frac{C_1}{2} (f(x)-f(x-h))^2.
\] 

The case $x \leq x_0-h$ follows by similar computation (one has to take $y=\theta (x+h) + (1-\theta)x$). Also, in the case $x \in [x_0-h,x_0+h]$ it is enough to take $y=\theta x_0 + (1-\theta)x$ and use the fact that $|x-y|=|\theta(x-x_0)|\leq h\theta$.

\end{proof}

\begin{proof}[Proof of Theorem \ref{convtau}]
We begin by showing that (a) implies (c). We do this in three steps.

\textit{Step 1.} We first show that it is enough to consider only the case when $f$ satisfies $|f'| \leq 1/C_\tau$. Since  $\inf f > -\infty$, then either there exists a point $y_0$ such that $|f'(y_0)|\leq 1/C_\tau$, or $f$ attains its minimum at some point $y_0$ (for example if $f(x)=2|x|/C_\tau$, only the second case is valid). Let us fix this point $y_0$. Moreover, let us define $y_-=\sup\{y \leq y_0: f'(y) \leq -1/C_\tau\}$ and  $y_+=\inf\{y \geq y_0: f'(y) \geq 1/C_\tau\}$ (if $f'<1/C_\tau$ then we adopt the notation $y_+=\infty$ and, similarly, $y_-=-\infty$ if $f'>-1/C_\tau$). Let us consider a function $g$ with $g=f$ on $[y_-,y_+]$  and with the derivative satisfying $g'=\textrm{sgn}(f')\min\{|f'|,1/C_\tau \}$ when $f' \geq 0$ and $g'=\max\{f',-1/C_\tau \}$ when $f'<0$ . Clearly $g \leq f$ and $g(x)=f(x)$ for $y_- \leq x \leq y_+$.

We now observe that $f \Box \vp = g \Box \vp$. To see this recall that $(f \Box \vp)(x)=\inf_{y}(f(y)+\vp(x-y))$. The function $\psi_x(y)=f(y)+\vp(x-y)$ satisfies $\psi_x'(y)=f'(y)-\vp'(x-y) \geq f'(y)-1/C_\tau$. Thus, for $y \geq y_+$ we have $\psi_x'(y) \geq 0$ and therefore $\psi_x$ is non-decreasing for $y \geq y_+$. Similarly, $\psi_x$ is non-increasing for $y \leq y_-$.  This means that  $\inf_{y}(f(y)+\vp(x-y))$ is attained at some point $y=y(x)\in [y_-,y_+]$. Thus,
\[
	(f \Box \vp)(x)=f(y(x))+\vp(x-y(x)) =g(y(x))+\vp(x-y(x)) = (g \Box \vp)(x). 
\] 
The last equality follows from the fact that the minimum of $y \mapsto g(y)+\vp(x-y)$ is achieved at the same point $y=y(x)$. Therefore, since $f \Box \vp =g \Box \vp$ and $g \leq f$, we have   
\[
	\left( \call{}{}{e^{f \Box \vp}}{\mu} \right) \left( \call{}{}{e^{-f}}{\mu} \right) \leq \left( \call{}{}{e^{g \Box \vp}}{\mu} \right) \left( \call{}{}{e^{-g}}{\mu} \right) 
\]
and $|g'| \leq 1/C_\tau$.

\textit{Step 2.} The inequality \eqref{deftau} stays invariant when we add a constant to the function $f$. Thus, we may and will assume that $f(0)=0$. Note that from the elementary inequality $4ab \leq (a+b)^2$ we have
\[
	4\left( \call{}{}{e^{f \Box \vp}}{\mu} \right) \left( \call{}{}{e^{-f}}{\mu} \right) \leq \left( \call{}{}{\left(e^{f \Box \vp}+e^{-f}\right)}{\mu}  \right)^2.
\]
Thus, it is enough to show that 
\[
	 \call{}{}{\left(e^{f \Box \vp}+e^{-f}\right)}{\mu} \leq 2.
\]

\textit{Step 3.}
Assume $|f'| \leq 1/C_\tau$. Take $C_1=17/(1-\lambda)^2$, $C_\tau=C_1h$ and $\vp(x) =\vp_0(x/C_\tau)$. By the convexity of $\vp_0$ we get $\vp(x)  \leq \frac{1}{C_1}\vp_0(x/h)$, since $C_1>1$. Thus, by Lemma \ref{bound} we get $f \Box \vp \leq f(x)- \frac12 C_1 |(Df)(x)|^2$. By the mean value theorem $|(Df)(x)|/h \leq 1/C_\tau$. Therefore, $\frac12 C_1 |(Df)(x)|^2 \leq \frac12 C_1 (\frac{h}{C_\tau})^2=1/2C_1$. Let $\alpha(C_1)=2C_1(1-\exp({-\frac{1}{2C_1}}))$. The convexity of the exponential function yields $e^{-s} \leq 1-\alpha(C_1)s$, $s \in [0,1/2C_1]$. Therefore, 
\begin{align*}
	  \call{}{}{\left(e^{f \Box \vp}+e^{-f}\right)}{\mu} & \leq \call{}{}{\left(e^{f(x)- \frac12 C_1 |(Df)(x)|^2}+e^{-f(x)}\right)}{\mu(x)} \\
	  & \leq \call{}{}{\left(e^{f(x)}\left(1- \frac12 C_1 \alpha(C_1) |(Df)(x)|^2 \right)+e^{-f(x)}\right)}{\mu(x)}.
\end{align*}
Therefore, since $e^f+e^{-f}-2=(e^{f/2}-e^{-f/2})^2$, we are to prove that
\[
	 \call{}{}{\left(e^{f/2} - e^{-f/2}  \right)^2}{\mu} \leq  \frac{C_1}{2}  \alpha(C_1) \call{}{}{ e^{f(x)} |(Df)(x)|^2 }{\mu(x)}.
\]
From Lemma \ref{bobex} this inequality is true whenever $\frac12 C_1 \alpha(C_1) \geq \frac{8}{(1-\lambda)^2}$. It suffices to observe that
\[
	\frac12 C_1 \alpha(C_1) = C_1^2 \left( 1- e^{-\frac{1}{2C_1}} \right) \geq C_1^2 \left( 1- \frac{1}{1+\frac{1}{2C_1}} \right)  = \frac{C_1}{2+\frac{1}{C_1}} \geq \frac{C_1}{2+\frac18} = \frac{8}{(1-\lambda)^2}. 
\]

We now sketch the proof of the fact that (c) implies (b). Due to the standard approximation argument one can assume that $f$ is a convex $C^2$ smooth function with bounded first and second derivative (note that in the definition of the convex Poincar\'e inequality we assumed that $f'$ is bounded). Consider the function $f_\e=\e f$. The infimum of $\psi_x(y)=\vp(y)+\e f(x-y)$ is attained at the point $y$ satisfying the equation $\psi_x'(y)=\vp'(y)-\e f'(x-y)=0$. Note that $\psi_x'$ is strictly increasing on the interval $[-C_\tau,C_\tau]$. If $\e$ is sufficiently small, it follows that the above equation has a unique solution $y_x$ and that $y_x \in [-C_\tau,C_\tau]$. Thus, $y_x=C_\tau^2 \e f'(x-y_x)$. This implies $y_x=\e C_\tau^2 f'(x)+o(\e)$. We get 
\begin{align*}
	f \Box \vp(x)& =\vp(y_x)+\e f(x-y_x)=\frac{1}{2C_\tau^2}y_x^2 + \e f(x-\e C_\tau^2 f'(x)) +o(\e^2) \\
	& = \frac12 \e^2  C_\tau^2 f'(x)^2 + \e f(x) - \e^2 C_\tau^2 f'(x)^2 + o(\e^2) 
\\ &	 = \e f(x) - \frac12  \e^2 C_\tau^2 f'(x)^2 +o(\e^2).
\end{align*}
Therefore, from the infimum convolution inequality we get
\[
	\left( \call{}{}{e^{\e f(x) - \frac12  \e^2 C_\tau^2 f'(x)^2+o(\e^2)}}{\mu(x)} \right) \left( \call{}{}{e^{-\e f}}{\mu} \right) \leq 1.
\]
Testing \eqref{deftau} with $f(x)=|x|/C_\tau$ one gets that $\call{}{}{e^\vp}{\mu}<\infty$ and therefore $\call{}{}{e^{|x|/C_\tau}}{\mu}<\infty$. Also, there exists a constant $c>0$ such that $|f(x)|\leq c(1+|x|)$, $x \in \mb{R}$. As a consequence, after some additional technical steps, one can consider the Taylor expansion of the above quantities in $\e=0$. This gives
\[
	\left( \call{}{}{\left(1+\e f -\frac12 \e^2 C_\tau^2 f'^2 +\frac12 \e^2 f^2 + o(\e^2)\right)}{\mu} \right) \left( \call{}{}{\left( 1- \e f + \frac12 \e^2 f^2 + o(\e^2) \right)}{\mu} \right) \leq 1.
\] 
Comparing the terms in front of $\e^2$ leads to 
\[
	\call{}{}{f^2}{\mu} - \left(\call{}{}{f}{\mu}\right)^2 \leq \frac12 C_\tau^2 \call{}{}{f'^2}{\mu}.
\] 
This is exactly the Poincar\'e inequality with constant $\frac12 C_\tau^2$.

We show that (b) implies (a). Suppose that a symmetric Borel probability measure $\mu$ on $\R$ satisfies the convex Poincar\'e inequality with a constant $C_p$. Consider the function $f_u(x) = \max\{x-u,0\}$, $u \geq 0$. We have
\[
\call{\R}{}{|f_u'(x)|^2}{\mu (x)} = \mu([u,\infty)).
\]
Since $f_u(y)=0$ for $y \leq 0$ and $\mu((-\infty,0]) \geq 1/2$, one gets
\begin{align*}
\textrm{Var}_\mu f_u & = \frac{1}{2}\int_\R\int_\R \left( f_u(x) - f_u(y) \right)^2 \dd\mu (x) \dd\mu(y) \geq \frac{1}{2}\int_\R\int_{-\infty}^0 \left( f_u(x) - f_u(y) \right)^2 \dd\mu (x) \dd\mu(y)  \\
& \geq    \frac{1}{4} \int_\R \left( f_u(x) \right)^2 \dd \mu (x) \geq  \frac{1}{4} \int_{u+\sqrt{8C_p}}^\infty \left( f_u(x) \right)^2 \dd \mu (x) \geq  2C_p\mu\left(\left[u + \sqrt{8C_p}, \infty\right)\right).
\end{align*}
These two observations, together with Poincar\'e inequality, yield that $\mu \in \mc{M}(\sqrt{8C_p},1/2)$.
\end{proof}

\section{Concentration properties}\label{corro}

We show that the convex property ($\tau$) implies concentration for convex sets.

\begin{prop}\label{gencon}
Suppose that a Borel probability measure $\mu$ on $\mb{R}^n$ satisfies the convex property $(\tau)$ with a non-negative cost function $\vp$, restricted to the family of convex functions. Let $B_\vp(t)=\{x \in \mb{R}^n: \ \vp(x) \leq t\}$. Then for any convex set $A$ we have
\[
	\mu(A+B_\vp(t)) \geq 1- \frac{e^{-t}}{\mu(A)}.
\]
\end{prop}
The proof of this proposition is similar to the proof of Proposition 2.4 in \cite{LW}. We recall the argument.

\begin{proof}
Let $f=0$ on $A$ and $f=\infty$ outside of $A$. Note that $f$ is convex (to avoid working with functions having values $+\infty$ one can consider a family of convex functions $f_n= n \dist(A,x)$ and take $n \to \infty$). Suppose that $(f \Box \vp)(x) \leq t$. Then there exists $y \in \mb{R}^n$ such that $f(y)+\vp(x-y) \leq t$. Thus, $y \in A$ and $x-y \in B_\vp(t)$. Therefore $x \in A+B_\vp(t)$. It follows that $x \notin A+B_\vp(t)$ implies $(f \Box \vp)(x) > t$. Applying the infimum convolution inequality we get 
\[
	e^t (1-\mu(A+B_\vp(t))\cdot \mu(A) \leq \left( \call{}{}{e^{f \Box \vp}}{\mu}\right) \left( \call{}{}{e^{-f}}{\mu} \right) \leq 1.
\]    
Our assertion follows.
\end{proof}

We are ready to derive the two-level concentration for convex sets.

\begin{proof}[Proof of Corollary \ref{corr1}]
The argument is similar to \cite[Corollary 1]{M}. Due to Corollary \ref{tensor}, $\mu=\mu_1 \otimes \ldots \otimes \mu_n$ satisfies property ($\tau$) with the cost function $\vp(x)=\sum_{i=1}^n \vp_0(x_i/C_\tau)$. Suppose that $\vp(x)\leq t$. Define $y,z \in \mb{R}^n$ in the following way. Take $y_i=x_i$ if $|x_i|\leq C_\tau$ and $y_i=0$ otherwise. Take $z_i=x_i$ if $|x_i|>C_\tau$ and $z_i=0$ otherwise. Then $x=y+z$. Moreover, 
\[
	 \sum_{i=1}^n \vp(y_i/C_\tau) +  \sum_{i=1}^n \vp(z_i/C_\tau) =   \sum_{i=1}^n \vp(x_i /C_\tau) \leq t.
\]
In particular $|y|_2^2 \leq 2C_\tau^2 t$ and  $t \geq \sum_{i=1}^n \vp(z_i/C_\tau) \geq \frac12 |z|_1 /C_\tau$, since $|z_i|-\frac12 \geq \frac12 |z_i|$ for $|z_i| \geq 1$. This gives $x \in \sqrt{2t}C_\tau B_2^n + 2 t C_\tau B_1^n$. Our assertion follows from Proposition \ref{gencon}.  
\end{proof}

Finally, we prove concentration for convex Lipschitz functions.  

\begin{proof}[Proof of Corollary \ref{corr2}]
The proof of \eqref{ineq1} is similar to the proof of Propositon 4.18 in \cite{L}. Let us define a convex set $A=\{f \leq \med_\mu f\}$ and observe that $\mu(A) \geq 1/2$. Moreover,
\[
	A+C_\tau(\sqrt{2t}B_2^n + 2t B_1^n) \subset \{f \leq \med_\mu f + C_\tau(a\sqrt{2t} + 2b t  )  \}.
\]
Applying Corollary \ref{corr1} we get
\[
	\mu\left( \left\{  f > \med_\mu f +  C_\tau(a\sqrt{2t} + 2bt  )  \right\} \right) \leq 2 e^{-t}, \qquad t \geq 0.
\]
Take $s= C_\tau(a\sqrt{2t} + 2b t  )$ and $r=s/C_\tau$. Suppose that $\frac{r}{b} \leq \frac{r^2}{a^2}$. Then 
\[
	a\sqrt{2t} + 2t b = r \geq \frac12 \sqrt{a^2r/b} +  \frac12 r = a\sqrt{2 \frac{r}{8b}} + 2b \frac{r}{4b}\geq a\sqrt{2 \frac{r}{8b}} + 2b \frac{r}{8b}.  	
\]
By the monotonicity of $x \mapsto a\sqrt{2x}+2xb$, $x \geq 0$ it follows that $\frac18 \min\{\frac{r}{b}, \frac{r^2}{a^2}\} = \frac{r}{8b} \leq t$. 
Moreover, if $\frac{r}{b} \geq \frac{r^2}{a^2}$, then 
\[
	a\sqrt{2t} + 2t b = r \geq   \frac12 r +  \frac{br^2}{2 a^2} = a \sqrt{2 \frac{r^2}{8a^2}} +  2b\frac{r^2}{4 a^2} \geq   a \sqrt{2 \frac{r^2}{8a^2}} +  2b \frac{r^2}{8 a^2}  	.
\]
Therefore, $\frac18 \min\{\frac{r}{b}, \frac{r^2}{a^2}\} = \frac{r^2}{8a^2} \leq t$. Thus,
\[
	\mu\left( \left\{  f > \med_\mu f +  r C_\tau  \right\} \right)  \leq 2 e^{-t} \leq 2 \exp\left( -\frac18 \min\left\{\frac{r}{b}, \frac{r^2}{a^2}\right\} \right), \qquad t \geq 0.
\]

For the proof of \eqref{ineq2} we follow \cite{T}. Define a convex set $B=\{f < \med_\mu f -  C_\tau(a\sqrt{2t} + 2bt  ) \}$ with $t\geq 0$. It follows that 
\[
	A+C_\tau(\sqrt{2t}B_2^n + 2t B_1^n) \subset \{f < \med_\mu f \}
\]
and thus Corollary \ref{corr1} yields
\[
	\frac12 \geq \mu\left( A+C_\tau(\sqrt{2t}B_2^n + 2t B_1^n) \right) \geq 1-\frac{e^{-t}}{\mu(B)}.
\]
Therefore $\mu(B)\leq 2e^{-t}$. To finish the proof we proceed as above.
\end{proof}

\section*{Acknowledgements}

We would like to thank Rados\l aw Adamczak for his useful comments regarding concentration estimates for Lipschitz functions. 

This research was supported in part by the Institute for Mathematics and its Applications with funds provided by the National Science Foundation.

\end{document}